\documentclass[10pt]{amsart}
\usepackage{amsmath,amscd}
\usepackage{amsbsy}
\usepackage{amssymb}
\usepackage{amscd,amsthm}
\usepackage[all,cmtip]{xy}

\newtheorem{thm}{Theorem}
\numberwithin{thm}{section}
\newtheorem{lem}[thm]{Lemma}
\newtheorem{prop}[thm]{Proposition}
\newtheorem{cor}[thm]{Corollary}
\newtheorem{exam}[thm]{Example}
\newtheorem{rema}[thm]{Remark}

\newtheorem{defi}[thm]{Definition}

\newtheorem*{thm2}{Theorem}
\newtheorem*{cor2}{Corollary}

\begin{document}
\begin{center}
\huge{Ulrich bundles on Brauer--Severi varieties}\\[1cm]
\end{center}

\begin{center}

\large{Sa$\mathrm{\check{s}}$a Novakovi$\mathrm{\acute{c}}$}\\[0,4cm]
{\small December 2019}\\[0,3cm]
\end{center}

\noindent{\small \textbf{Abstract}. 
We prove the existence of Ulrich bundles on any Brauer--Severi variety. In some cases, the minimal possible rank of the obtained Ulrich bundles equals the period of the Brauer--Severi variety. Moreover, we find a formula for the rank of an Ulrich bundle involving the period of the considered Brauer--Severi variety $X$, at least if $\mathrm{dim}(X)=p-1$ for an odd prime $p$. This formula implies that the rank of any Ulrich bundle on such a Brauer--Severi variety $X$ must be a multiple of the period. \\

\section{Introduction}
Let $H$ a very ample line bundle on a variety $X$. In \cite{ESWV} the authors defined an \emph{Ulrich bundle} for $(X,H)$ to be a vector bundle $\mathcal{E}$ satisfying $h^q(X,\mathcal{E}(-iH))=0$ for each $q\in\mathbb{Z}$ and $1\leq i\leq \mathrm{dim}(X)$. Ulrich bundles are somehow the "nicest" arithmetically Cohen--Macaulay sheaves which are important to understand, since they give a measurement of the complexity of the variety. It was conjectured in \cite{ESWV} that any projective variety carries an Ulrich bundle. Moreover, it is asked for the smallest possible rank of such a bundle. This conjecture is a wide open problem, and we know a few result at the present. Varieties known to carry Ulrich sheaves include curves and Veronese varieties \cite{ESWV}, \cite{HV}, complete intersections \cite{BHU1V}, generic linear determinantal varieties \cite{BHUV}, Segre varieties \cite{CMRPV}, rational normal scrolls \cite{MRV}, Grassmannians \cite{CMRV}, some flag varieties \cite{CMRV}, \cite{CHWV}, generic $\mathrm{K}$3 surfaces \cite{AFOV}, abelian surfaces \cite{BEV}, Enriques surfaces \cite{BNV} and ruled surfaces \cite{ACMRV}. Among others, in \cite{ESWV} it is proved that any curve $C\subset\mathbb{P}^n$ has a rank-$2$ Ulrich sheaf, provided the base field of the curve is infinite. Examples may be pointless conics defined by $x^2+y^2+z^2=0$. Recall that a scheme $X$ of finite type over a field $k$ is called \emph{Brauer--Severi variety} if $X\otimes_k\bar{k}\simeq \mathbb{P}^n$. Via Galois cohomology, isomorphism classes of Brauer--Severi varieties over $k$ are in one-to-one correspondence with isomorphism classes of central simple algebras over $k$. Moreover, Brauer--Severi varieties (respectively the corresponding central simple algebras) have important invariants called period, index and degree (see Section 2 for all the details). Pointless conics such as those defined by $x^2+y^2+z^2=0$ over a subfield of $\mathbb{R}$ are Brauer--Severi varieties, and since these admit rank-$2$ Ulrich bundles by \cite{ESWV}, it is natural to ask about Ulrich bundles over Brauer--Severi varieties of higher dimension.
In the case of non-split Brauer--Severi curves there are always Ulrich bundles of rank two (see Proposition 5.4). It is an easy observation that non-split Brauer--Severi varieties cannot have Ulrich line bundles (see Proposition 5.3). So the minimal rank of an Ulrich bundle on a Brauer--Severi curve is one or two, depending whether the curve is split or not. 

In general, we show that there exist always Ulrich bundles on any polarized Brauer--Severi variety (see Propositions 5.1 and 5.2). So the existence of Ulrich bundles for Brauer--Severi varieties has been solved completely. In particular, we show that the rank of the (obtained) Ulrich bundle is divisible by the period (see Proposition 5.1). In view of the fact that not much is known about the minimal rank of Ulrich bundles for $(\mathbb{P}^n,\mathcal{O}_{\mathbb{P}^n}(d))$, it seems to be a challenging problem to determine the minimal rank of Ulrich bundles on Brauer--Severi varieties. 
Recall that on a $n$-dimensional Brauer--Severi variety $X$, corresponding to a central simple algebra $A$, there are indecomposable vector bundles $\mathcal{V}_i$, $i\in\mathbb{Z}$, satisfying $\mathcal{V}_i\otimes_k \bar{k}\simeq \mathcal{O}_{\mathbb{P}^n}(i)^{\oplus d_i}$, where $d_i=\mathrm{ind}(A^{\otimes i})$ denotes the index of the central simple algebra $A^{\otimes i}$. These $\mathcal{V}_i$, $i\in\mathbb{Z}$, are unique up to isomorphism (see \cite{NV}) and one has $\mathrm{End}(\mathcal{V}_i)\simeq D_i$, where $D_i$ denotes the central division algebra that is Brauer-equivalent to $A^{\otimes i}$. It is well known that the set $\mathcal{V}_0,...,\mathcal{V}_n$ is a full weak exceptional collection on a Brauer--Severi variety of dimension $n$ (see \cite{BERV} or \cite{ORV}, Example 1.17). Calculating the right dual of this collection (see Proposition 4.8) and using a generalization of Beilinson's spectral sequence for Brauer--Severi varieties (see Corollary 4.10), we can prove the following:
\begin{thm2}[Theorem 5.10]
	Let $n\neq 1$ and $X$ a $n$-dimensional Brauer--Severi variety of period $p$. Set $d_n=\mathrm{rk}(\mathcal{V}_n)$. If $\mathcal{E}$ is an Ulrich bundle for $(X,\mathcal{O}_X(pd))$, then there is an exact sequence
	\begin{eqnarray}
		0\rightarrow \mathcal{V}^{\oplus a_{-n}}_{-n}\rightarrow \cdots \rightarrow \mathcal{V}^{\oplus a_{-n+1}}_{-n+1}\rightarrow \mathcal{V}^{\oplus a_{-1}}_{-1}\rightarrow \mathcal{E}\otimes\mathcal{O}_X(-pd)^{\oplus d^2_n}\rightarrow 0.
	\end{eqnarray}
\end{thm2}  
With the help of Theorem 5.10, one can try to determine the rank of an Ulrich bundle. In the special case where the Brauer--Severi variety has dimension $p-1$ for an odd prime $p$, we prove:
\begin{thm2}[Theorem 5.12]
	Let $X$ be a non-split Brauer--Severi variety of dimension $p-1$, where $p$ denotes an odd prime, and $\mathcal{E}$ an Ulrich bundle for $(X,\mathcal{O}_X(p))$. Then there are integers $c_{p-1-j}>0$, $j=1,...,p-2$, such that
	\begin{eqnarray*}
		\mathrm{rk}(\mathcal{E})\cdot ((p-1)!+\prod^{p-2}_{i=1}{(ip-1)})=(p\cdot(p-1)!)(\sum^{p-2}_{j=1}{(-1)^{j+1}c_{p-1-j}}).
	\end{eqnarray*}
	In particular, $p$ divides $\mathrm{rk}(\mathcal{E})\cdot ((p-1)!+\prod^{p-2}_{i=1}{(ip-1)})$.
\end{thm2}
So Theorem 5.12 implies that if $p$ does not divide $((p-1)!+\prod^{p-2}_{i=1}{(ip-1)})$, the rank $\mathrm{rk}(\mathcal{E})$ of an Ulrich bundle must be a multiple of $p$. And indeed, applying Wilson's theorem, we find:
\begin{cor2}[Corollary 5.13]
	Let $X$ be a non-split Brauer--Severi variety as in Theorem 5.12. Then the rank of any Ulrich bundle for $(X,\mathcal{O}_X(p))$ must be a multiple of $p$.  
\end{cor2}
Using a generalized Hartshorne--Serre correspondence as stated in \cite{BAV}, we can show that in certain cases a Brauer--Severi variety associated to a central simple division algebra of index $4$ and period $2$ admits a unique rank two Ulrich bundle. More precise, we show:
\begin{thm2}[Theorem 5.15]
	Let $X$ be a Brauer--Severi variety of dimension $3$ where $\mathcal{O}_X(2)$ exists. Suppose there is a smooth geometrically connected genus one curve $C$ on $X$ such that the restriction map $H^0(X,\mathcal{O}_X(2))\rightarrow H^0(C,\mathcal{O}_C(2))$ is bijective. Then there is an  unique Ulrich bundle of rank two for $(X,\mathcal{O}_X(2))$.
\end{thm2}
In Example 5.17 we show that there exist smooth genus one curves on a Brauer--Severi threefold over $\mathbb{R}$ satisfying the hypothesis of Theorem 5.15. A consequence of the theorem is the following result.
\begin{cor2}[Corollary 5.16]
	Let $X$ be as in Theorem 5.15. Then there is a Ulrich bundle of rank $12$ for $(X,\mathcal{O}_X(2d))$ for any $d\geq 1$.
\end{cor2}
The significance of Corollary 5.15 is explained in Remark 5.17. We recall that in \cite{BV} it is proved that any Fano threefold of index two carries a special rank two Ulrich bundle. In particular, there is always a special rank two Ulrich bundle for $(\mathbb{P}^3,\mathcal{O}_{\mathbb{P}^3}(2))$. Since there is no rank one Ulrich bundle for $(\mathbb{P}^3,\mathcal{O}_{\mathbb{P}^3}(2))$, the minimal rank of an Ulrich bundle for $(\mathbb{P}^3,\mathcal{O}_{\mathbb{P}^3}(2))$ is two. The period of $\mathbb{P}^3$, considered as a trivial Brauer--Severi variety, is one. Notice that in \cite{LV} it is shown that there is a unique rank two Ulrich bundle for $(\mathbb{P}^2,\mathcal{O}_{\mathbb{P}^2}(2))$. These results together with the results obtained in this paper lead us to formulate some questions concerning Ulrich bundles on Brauer--Severi varieties. So let $X$ be an arbitrary Brauer--Severi variety of period $\mathrm{per}(X)=p$.

\begin{itemize}
	\item[1)] Is there always an Ulrich bundle of rank $\mathrm{per}(X)$ for $(X,\mathcal{O}_X(p))$ ?
	\item[2)] Is the minimal rank of an Ulrich bundle for $(X,\mathcal{O}_X(p))$ exactly $\mathrm{per}(X)=p$ ?
	\item[3)] How does the minimal rank of an Ulrich bundle for $(X,\mathcal{O}_X(p\cdot d))$ depend on $d$ ?
	\item[4)]	Suppose the minimal rank of an Ulrich bundle does not equal $\mathrm{per}(X)$. Is there a formula involving the invariants period, index and degree that calculates the minimal rank ?  
\end{itemize}
{\small \textbf{Acknowledgement}. I thank Lucian B\u adescu for answering questions about the generalized Hartshorne--Serre correspondence.  This research was conducted in the framework of the research training group GRK 2240: Algebro-geometric Methods in Algebra, Arithmetic and Topology, which is funded by the $\mathrm{DFG}$.}\\

\noindent{\small \textbf{Convetions}. Throughout this work $k$ is an arbitrary field. Moreover, $\bar{k}$ denotes an algebraic closure and $\bar{\mathcal{E}}$ the base change of a vector bundle $\mathcal{E}$ over $k$ to $\bar{k}$. The dimension of the cohomology group $H^i(X,\mathcal{F})$ as $k$ vector space is abbreviated by $h^i(\mathcal{F})$. Similary, we write $\mathrm{ext}^i(\mathcal{F},\mathcal{G})$ for $\mathrm{dim}\mathrm{Ext}^i(\mathcal{F},\mathcal{G})$. For $X\times Y$, let $f$ and $g$ be the projections to $X$ and $Y$ respectively. The tensor product $f^*\mathcal{F}\otimes g^*\mathcal{G}$ will be denoted by $\mathcal{F}\boxtimes\mathcal{G}$.}


\section{Brauer--Severi varieties}
We recall the basics of Brauer--Severi varieties and central simple algebras and refer to \cite{GSV} and references therein for details. A \emph{Brauer--Severi variety} of dimension $n$ is a scheme $X$ of finite type over $k$ such that $X\otimes_k L\simeq \mathbb{P}^n$ for a finite field extension $k\subset L$. This definition of a Brauer--Severi variety is equivalent to the definition given in the introduction (see \cite{GSV}, Remark 5.12). A field extension $k\subset L$ such that $X\otimes_k L\simeq \mathbb{P}^n$ is called \emph{splitting field} of $X$. Clearly, the algebraic closure $\bar{k}$ is a splitting field for any Brauer--Severi variety. One can show that a Brauer--Severi variety always splits over a finite separable field extension of $k$ (see \cite{GSV}, Corollary 5.1.4). By embedding the finite separable splitting field into its Galois closure, a Brauer--Severi variety splits over a finite Galois extension of the base field $k$ (see \cite{GSV}, Corollary 5.1.5). It follows from descent theory that $X$ is projective, integral and smooth over $k$. If the Brauer--Severi variety $X$ is already isomorphic to $\mathbb{P}^n$ over $k$, it is called \emph{split}, otherwise it is called \emph{non-split}. There is a well-known one-to-one correspondence between Brauer--Severi varieties and central simple $k$-algebras. Recall that an associative $k$-algebra $A$ is called \emph{central simple} if it is an associative finite-dimensional $k$-algebra that has no two-sided ideals other than $0$ and $A$ and if its center equals $k$. If the algebra $A$ is a division algebra, it is called \emph{central division algebra}. For instance, a Brauer--Severi curve is associated to a quaternion algebra (see \cite{GSV}). Central simple $k$-algebras can be characterized by the following well-known fact (see \cite{GSV}, Theorem 2.2.1): $A$ is a central simple $k$-algebra if and only if there is a finite field extension $k\subset L$ such that $A\otimes_k L \simeq M_n(L)$ if and only if $A\otimes_k \bar{k}\simeq M_n(\bar{k})$.

The \emph{degree} of a central simple algebra $A$ is now defined to be $\mathrm{deg}(A):=\sqrt{\mathrm{dim}_k A}$. According to the Wedderburn Theorem, for any central simple $k$-algebra $A$ there is an integer $n>0$ and a division algebra $D$ such that $A\simeq M_n(D)$. The division algebra $D$ is also central and unique up to isomorphism. Now the degree of the unique central division algebra $D$ is called the \emph{index} of $A$ and is denoted by $\mathrm{ind}(A)$. It can be shown that the index is the smallest among the degrees of finite separable splitting fields of $A$ (see \cite{GSV}, Corollary 4.5.9). Two central simple $k$-algebras $A\simeq M_n(D)$ and $B\simeq M_m(D')$ are called \emph{Brauer equivalent} if $D\simeq D'$. Brauer equivalence is indeed an equivalence relation and one defines the \emph{Brauer group} $\mathrm{Br}(k)$ of a field $k$ as the group whose elements are equivalence classes of central simple $k$-algebras and group operation being the tensor product. It is an abelian group with inverse of the equivalence class of $A$ given by the equivalence class of $A^{op}$. The neutral element is the equivalence class of $k$. The order of a central simple $k$-algebra $A$ in $\mathrm{Br}(k)$ is called the \emph{period} of $A$ and is denoted by $\mathrm{per}(A)$. It can be shown that the period divides the index and that both, period and index, have the same prime factors (see \cite{GSV}, Proposition 4.5.13). Denoting by $\mathrm{BS}_n(k)$ the set of all isomorphism classes of Brauer--Severi varieties of dimension $n$ and by $\mathrm{CSA}_{n+1}(k)$ the set of all isomorphism classes of central simple $k$-algebras of degree $n+1$, there is a canonical identification $\mathrm{CSA}_{n+1}(k)=\mathrm{BS}_n(k)$ via non-commutative Galois cohomology (see \cite{GSV} for details). Hence any $n$-dimensional Brauer--Severi variety $X$ corresponds to a central simple $k$-algebra of degree $n+1$. In view of the one-to-one correspondence between Brauer--Severi varieties and central simple algebras one can also speak about the period of a Brauer--severi variety $X$. It is defined to be the period of the corresponding central simple $k$-algebra $A$.

Geometrically, the period of a Brauer--Severi variety $X$ can be interpreted as the smallest positive integer $p$ such that $\mathcal{O}_X(p)$ exists on $X$. In other words, if $X\otimes_k L\simeq \mathbb{P}^n$, then $p$ is the smallest positive integer such that $\mathcal{O}_{\mathbb{P}^n}(p)$ descends to a line bundle on $X$. 
Moreover, the Picard group of $X$ is isomorphic to $\mathbb{Z}$ and is generated by $\mathcal{O}_X(p)$. In the present paper, we make use of the following fact concerning embeddings of a Brauer--Severi variety.
\begin{thm}[\cite{LIV}, Corollary 3.6]
	Let $X$ be a Brauer--Severi variety of period $p$ over $k$. Then any line bundle $\mathcal{O}_X(pd)$ for $d\geq 1$ gives rise to an embedding
	\begin{eqnarray*}
		\phi_{pd}\colon X\longrightarrow \mathbb{P}^{N}, \;\; \text{where}\;\; N=\binom{\mathrm{dim}(X)+pd}{pd}.
	\end{eqnarray*} 
	After base change to a splitting field $L$ of $X$, this embedding becomes the $dp$-uple Veronese embedding of $X\otimes_k L=\mathbb{P}_L^{\mathrm{dim}(X)}$ into $\mathbb{P}_L^{N}$.
\end{thm}

\section{Generalities on Ulrich bundles}
Let $H$ a very ample line bundle on a variety $X$. Recall that in \cite{ESWV} a vector bundle $\mathcal{E}$ on $X$ is defined to be an \emph{Ulrich bundle} for $(X,H)$ if it satisfies $h^q(X,\mathcal{E}(-iH))=0$ for each $q\in\mathbb{Z}$ and $1\leq i\leq \mathrm{dim}(X)$. Although there are further properties of Ulrich bundles, we will list only those needed in the present paper. We refer the reader to \cite{BV}, \cite{ESWV} for details.

\begin{lem}[\cite{BV}, (3.5)] 
	Let $\mathcal{E}$ and $\mathcal{F}$ be Ulrich bundles for $(X,\mathcal{O}_X(1))$ and $(Y,\mathcal{O}_Y(1))$ and put $n=\mathrm{dim}(X)$. Then $\mathcal{E}\boxtimes \mathcal{F}(n)$ is an Ulrich bundle for $(X\times Y,\mathcal{O}_X(1)\boxtimes \mathcal{O}_Y(1))$.
\end{lem}
\begin{lem}[\cite{BV}, (3.6)]
	Let $\pi\colon X\rightarrow Y$ be a finite surjective morphism, $\mathcal{L}$ a very ample line bundle on $Y$ and $\mathcal{E}$ a vector bundle on $X$. Then $\mathcal{E}$ is an Ulrich bundle for $(X,\pi^*\mathcal{L})$ if and only if $\pi_*\mathcal{E}$ is an Ulrich bundle for $(Y,\mathcal{L})$. 
\end{lem}
In the special case where $X$ is a Brauer--Severi variety, we also have:
\begin{prop}
	Let $X$ be a Brauer--Severi variety of period $p$ with splitting field $L$ and $d\geq 1$. A vector bundle $\mathcal{E}$ is an Ulrich bundle for $(X,\mathcal{O}_X(pd))$ if and only if $\mathcal{E}\otimes_k \bar{k}$ is an Ulrich bundle for $(X\otimes_k L,\mathcal{O}_{X\otimes_k L}(pd))$.
\end{prop}
\begin{proof}
	The Brauer--Severi variety $X$ is embedded into $\mathbb{P}^N$ via $\mathcal{O}_X(pd)$ with the morphism $\phi_{pd}$ given in Theorem 2.1. Note that $H^i(X,\mathcal{F})\otimes_k E\simeq H^i(X\otimes_k E,\mathcal{F}\otimes_k E)$ for any coherent sheaf $\mathcal{F}$ and any field extension $k\subset E$. The assertion then follows from the fact that $\mathcal{O}_X(pd)$ is very ample if and only if $\mathcal{O}_{X\otimes_k L}(pd)$ is (see \cite{LIV}, Lemma 3.2 (2)).
\end{proof}

\section{Exceptional collections and Beilinson type spectral sequence}
Let $\mathcal{D}$ be a triangulated category and $\mathcal{C}$ a triangulated subcategory. The subcategory $\mathcal{C}$ is called \emph{thick} if it is closed under isomorphisms and direct summands. Note that there are different definitions of thick subcategory in the literature. For a subset $A$ of objects of $\mathcal{D}$ we denote by $\langle A\rangle$ the smallest full thick subcategory of $\mathcal{D}$ containing the elements of $A$. For a smooth projective variety $X$ over $k$, we denote by $D^b(X)$ the bounded derived category of coherent sheaves on $X$. Moreover, if $B$ is an associated $k$-algebra, we write $D^b(B)$ for the bounded derived category of finitely generated left $B$-modules.
\begin{defi}
	\textnormal{Let $A$ be a division algebra over $k$, not necessarily central. An object $\mathcal{E}^{\bullet}\in D^b(X)$ is called \emph{$A$-exceptional} if $\mathrm{End}(\mathcal{E}^{\bullet})=A$ and $\mathrm{Hom}(\mathcal{E}^{\bullet},\mathcal{E}^{\bullet}[r])=0$ for $r\neq 0$. By \emph{weak exceptional object}, we mean $A$-exceptional for some division algebra $A$ over $k$. 
		If $A=k$, the object $\mathcal{E}^{\bullet}$ is called \emph{exceptional}. } 
\end{defi}
\begin{defi}
	\textnormal{A totally ordered set $\{\mathcal{E}^{\bullet}_0,...,\mathcal{E}^{\bullet}_n\}$ of weak exceptional objects on $X$ is called an \emph{weak exceptional collection} if $\mathrm{Hom}(\mathcal{E}^{\bullet}_i,\mathcal{E}^{\bullet}_j[r])=0$ for all integers $r$ whenever $i>j$. A weak exceptional collection is \emph{full} if $\langle\{\mathcal{E}^{\bullet}_0,...,\mathcal{E}^{\bullet}_n\}\rangle=D^b(X)$ and \emph{strong} if $\mathrm{Hom}(\mathcal{E}^{\bullet}_i,\mathcal{E}^{\bullet}_j[r])=0$ whenever $r\neq 0$. If the set $\{\mathcal{E}^{\bullet}_0,...,\mathcal{E}^{\bullet}_n\}$ consists of exceptional objects it is called \emph{exceptional collection}.}
\end{defi}
The notion of a full exceptional collection is a special case of what is called a semiorthogonal decomposition of $D^b(X)$. Recall that a full triangulated subcategory $\mathcal{D}$ of $D^b(X)$ is called \emph{admissible} if the inclusion $\mathcal{D}\hookrightarrow D^b(X)$ has a left and right adjoint functor. 
\begin{defi}
	\textnormal{Let $X$ be a smooth projective variety over $k$. A sequence $\mathcal{D}_0,...,\mathcal{D}_n$ of full triangulated subcategories of $D^b(X)$ is called \emph{semiorthogonal} if all $\mathcal{D}_i\subset D^b(X)$ are admissible and $\mathcal{D}_j\subset \mathcal{D}_i^{\perp}=\{\mathcal{F}^{\bullet}\in D^b(X)\mid \mathrm{Hom}(\mathcal{G}^{\bullet},\mathcal{F}^{\bullet})=0$, $\forall$ $ \mathcal{G}^{\bullet}\in\mathcal{D}_i\}$ for $i>j$. Such a sequence defines a \emph{semiorthogonal decomposition} of $D^b(X)$ if the smallest full thick subcategory containing all $\mathcal{D}_i$ equals $D^b(X)$.}
\end{defi}
For a semiorthogonal decomposition we write $D^b(X)=\langle \mathcal{D}_0,...,\mathcal{D}_n\rangle$.
\begin{rema}
	\textnormal{Let $\mathcal{E}^{\bullet}_0,...,\mathcal{E}^{\bullet}_n$ be a full weak exceptional collection on $X$. It is easy to verify that by setting $\mathcal{D}_i=\langle\mathcal{E}^{\bullet}_i\rangle$ one gets a semiorthogonal decomposition $D^b(X)=\langle \mathcal{D}_0,...,\mathcal{D}_n\rangle$.}
\end{rema}
\begin{exam}
	\textnormal{Let $X$ be a Brauer--Severi variety of dimension $n$ and $\mathcal{V}_i$, $i\in\mathbb{Z}$, the vector bundles from the introduction. Then the ordered set $\{\mathcal{V}_0,\mathcal{V}_1,...,\mathcal{V}_n\}$ is a full weak exceptional collection (see \cite{BERV} or \cite{ORV}, Example 1.17).}
\end{exam}
Let $\mathcal{E}^{\bullet}$ be an exceptional object in $D^b(X)$. For any object $\mathcal{F}^{\bullet}\in D^b(X)$ there is a \emph{left} and \emph{right mutation} with respect to $\mathcal{F}$ give by the distinguished triangles
\begin{eqnarray*}
	L_{\mathcal{E}^{\bullet}}\mathcal{F}\longrightarrow \mathrm{Hom}^{\bullet}(\mathcal{E}^{\bullet},\mathcal{F})\otimes \mathcal{E}^{\bullet}\longrightarrow \mathcal{F}\longrightarrow L_{\mathcal{E}^{\bullet}}\mathcal{F}[1],\\
	R_{\mathcal{E}^{\bullet}}\mathcal{F}[-1]\longrightarrow \mathcal{F}\longrightarrow \mathrm{Hom}^{\bullet}(\mathcal{F},\mathcal{E}^{\bullet})^*\otimes \mathcal{E}^{\bullet}\longrightarrow R_{\mathcal{E}^{\bullet}}\mathcal{F}.
\end{eqnarray*}
If $\{\mathcal{E}^{\bullet},\mathcal{F}^{\bullet}\}$ is an exceptional pair, then $\{L_{\mathcal{E}^{\bullet}}\mathcal{F}^{\bullet},\mathcal{E}^{\bullet}\}$ and $\{\mathcal{F}^{\bullet}, R_{\mathcal{F}^{\bullet}},\mathcal{E}^{\bullet}\}$ are exceptional pairs, too. 

Now let $\mathcal{E}_0,...,\mathcal{E}_n$ be a full weak exceptional collection of coherent sheaves. We call the two collections obtained in the following way
\begin{eqnarray*}
	\mathcal{E}^{\vee}_i:=L_{\mathcal{E}_0}L_{\mathcal{E}_2}\cdots L_{\mathcal{E}_{n-i-1}}\mathcal{E}_{n-i}:=L_{\mathcal{E}_1\cdots\mathcal{E}_{n-i-1}}\mathcal{E}_{n-i},\\
	^{\vee}\mathcal{E}_i:=R_{\mathcal{E}_n}R_{\mathcal{E}_{n-1}}\cdots R_{\mathcal{E}_{n-i+1}}\mathcal{E}_{n-i}:=R_{\mathcal{E}_n\cdots \mathcal{E}_{n-i+1} }\mathcal{E}_{n-i}  
\end{eqnarray*}
the \emph{left} respectively \emph{right dual} of $\mathcal{E}_0,...,\mathcal{E}_n$. Note that the left and right dual are again full weak exceptional collections. 
In a similar way, one can define left and right mutation for a pair $\{A,B\}$ of admissible subcategories of $D^b(X)$. In this case, performing left or right mutation commutes with finite field extension. If for instance $A=\langle \mathcal{G}_1,...,\mathcal{G}_r\rangle$ and $B=\langle\mathcal{H}_1,...,\mathcal{H}_s\rangle$ are generated by weak exceptional collections of coherent sheaves, one sets
\begin{eqnarray*}
	R_B\mathcal{F}^{\bullet}:=R_{\mathcal{H}_s}R_{\mathcal{H}_{s-1}}\cdots R_{\mathcal{H}_1}\mathcal{F}^{\bullet}.
\end{eqnarray*}
Obviously, if $B=\langle\mathcal{H}\rangle$ is generated by a single weak exceptional coherent sheaf, we have $R_B\mathcal{F}^{\bullet}=R_{\mathcal{H}}\mathcal{F}^{\bullet}$.
\begin{thm}
	Let $X$ be a smooth projective variety over a field $k$ and $\mathcal{E}_0,...,\mathcal{E}_n$ a full weak exceptional collection of coherent sheaves. Then for any coherent sheaf $\mathcal{F}$ there is a spectral sequence
	\begin{eqnarray}
		E^{p,q}_1=\mathrm{Ext}^q(R_{\mathcal{E}_n\cdots\mathcal{E}_{p+n+1}}\mathcal{E}_{p+n},\mathcal{F})\otimes \mathcal{E}_{p+n}\Rightarrow E^{p+q}=\begin{cases}\mathcal{F}& \text{if} \;\;\;\;\; p+q=0\\
			0& \text{otherwise}
		\end{cases}
	\end{eqnarray}
	The grading is bounded by $0\leq q\leq n$ and $-n\leq p\leq 0$.
\end{thm}
\begin{proof}
	The proof follows directly from general facts on performing successive left and right mutations and is analogous to the proof of Theorem 3.16 of \cite{CMR1V}. We sketch the proof. Let 
	\begin{center}
		$V^{\bullet}_i=\mathrm{Hom}^{\bullet}(R_{\mathcal{E}_n\cdots \mathcal{E}_{i+1} }\mathcal{E}_{i},\mathcal{F})=\mathrm{Hom}^{\bullet}(\mathcal{E}_i,L_{\mathcal{E}_{i+1}\cdots \mathcal{E}_{n} }\mathcal{F})$.
	\end{center}
	So the triangles defining consequent right and left mutations of $\mathcal{F}$ can be written as
	\begin{eqnarray*}
		V^{\bullet}_{\mu}\otimes \mathcal{E}_{\mu}[-1]\longrightarrow R_{\mathcal{E}_n\cdots \mathcal{E}_{\mu}}\mathcal{F}[-1]\longrightarrow R_{\mathcal{E}_n\cdots \mathcal{E}_{\mu-1}}\mathcal{F}\longrightarrow V^{\bullet}_{\mu}\otimes \mathcal{E}_{\mu},\\
		V^{\bullet}_{\mu}\otimes \mathcal{E}_{\mu}\longrightarrow L_{\mathcal{E}_{\mu+1}\cdots \mathcal{E}_n}\mathcal{F}[n]\longrightarrow L_{\mathcal{E}_{\mu}\cdots \mathcal{E}_n}\mathcal{F}[n+1]\longrightarrow V^{\bullet}_{\mu}\otimes \mathcal{E}_{\mu}[1].
	\end{eqnarray*}
	These triangles give rise to a complex
	\begin{eqnarray*}
		L^{\bullet}: 0\longrightarrow V^{\bullet}_0\otimes \mathcal{E}_0\longrightarrow V^{\bullet}_2\otimes \mathcal{E}_2\longrightarrow \cdots \longrightarrow V^{\bullet}_{n-1}\otimes \mathcal{E}_{n-1}\longrightarrow V^{\bullet}_n\otimes \mathcal{E}_n\longrightarrow 0 
	\end{eqnarray*}
	which is functorial in $\mathcal{F}$ (see \cite{GKV}, p.391). The left, respectively right, mutation produce a canonical Postnikov-system (see \cite{GKV}, p.390-392 or \cite{CMR1V}, p.86-87) which identifies $\mathcal{F}$ with a canonical right convolution of the above complex. Then, for an arbitrary linear covariant cohomological functor $\Phi^{\bullet}$, there exists a spectral sequence with $E_1$-term $E_1^{pq}=\Phi^q(L^p)$ which converges to $\Phi^{p+q}(\mathcal{F})$. In particular, if we take $\Phi^{\bullet}$ to be the cohomology functor wich takes a complex to its cohomology sheaf, we have
	\begin{eqnarray*}
		\Phi^{\beta}(\mathcal{F})=\begin{cases}\mathcal{F}& \text{for} \;\;\beta=0\\
			0& \text{otherwise}
		\end{cases}
	\end{eqnarray*}
	After an index change (see \cite{CMR1V}, p.87), this gives finally the desired spectral sequence.
\end{proof}
We recall the following fact, that is used frequently from now on.

\begin{prop}[\cite{NV}, Proposition 3.4]
	Let $X$ be a proper $k$-scheme and $\mathcal{F}$ and $\mathcal{G}$ two coherent sheaves. If $\mathcal{F}\otimes_k \bar{k}\simeq \mathcal{G}\otimes_k \bar{k}$, then $\mathcal{F}$ is isomorphic to $\mathcal{G}$.
\end{prop}

\begin{prop}
	Let $X$ be a $n$-dimensional Brauer--Severi variety over $k$. Then the right dual of the full weak exceptional collection $\mathcal{V}_0=\mathcal{O}_X,\mathcal{V}_1,...,\mathcal{V}_{n}$ is given by $^{\vee}\mathcal{V}_l=R_{\mathcal{V}_{n}\cdots\mathcal{V}_{n-l+1}}\mathcal{V}_{n-l}\simeq \wedge^l\mathcal{T}_X\otimes \mathcal{V}_{n-l}$ for $1\leq l\leq n$ and $^{\vee}\mathcal{V}_0\simeq \mathcal{V}_n$. 
\end{prop}
\begin{proof}
	Let $A$ be the degree $n+1$ central simple algebra corresponding to $X$. We note that $\mathcal{V}_1\otimes_k\bar{k}\simeq \mathcal{O}_{\mathbb{P}^n}(1)^{\oplus \mathrm{ind}(A)}$. As mentioned in Section 2, the index $\mathrm{ind}(A)$ divides $n+1$. Let $n+1=a_1\cdot \mathrm{ind}(A)$.  To prove our assertion, we consider the external powers of the Euler exact sequence on $X$
	\begin{eqnarray*}
		\begin{xy}
			\xymatrix{
				0\ar[r] &\mathcal{O}_X\ar[r] &\mathcal{V}_1^{\oplus a_1}\ar[r] &\mathcal{T}_X\ar[r]& 0,\\
				0\ar[r] & \mathcal{T}_X\ar[r] &\wedge^2(\mathcal{V}_1^{\oplus a_1})\ar[r] &\wedge^2\mathcal{T}_X\ar[r] & 0,\\
				& \vdots							& \vdots											&	\vdots \\
				0\ar[r] & \wedge^{n-1}\mathcal{T}_X\ar[r] &\wedge^{n}(\mathcal{V}_1^{\oplus a_1})\ar[r] &\mathcal{O}_X(n+1)\ar[r] & 0.				
			}
		\end{xy}
	\end{eqnarray*}
	By definition, we have $^{\vee}\mathcal{V}_0=\mathcal{V}_{n}$. We will show how to calculate $^{\vee}\mathcal{V}_1$ and $^{\vee}\mathcal{V}_2$ and left the remaining cases to the reader, since they are obtained inductively. So consider $^{\vee}\mathcal{V}_1= R_{\langle\mathcal{V}_{n}\rangle}\mathcal{V}_{n-1}$. So after base change to some finite splitting field $L$ of $X$, we have
	\begin{eqnarray*}
		(^{\vee}\mathcal{V}_1)_L= R_{\langle\mathcal{V}_{n}\rangle_L}(\mathcal{V}_{n-1})_L\simeq R_{\langle\mathcal{O}_{\mathbb{P}}(n)\rangle}(\mathcal{O}_{\mathbb{P}}(n-1)^{\oplus \mathrm{rk}(\mathcal{V}_{n-1})})\simeq R_{\mathcal{O}_{\mathbb{P}}(n)}(\mathcal{O}_{\mathbb{P}}(n-1)^{\oplus \mathrm{rk}(\mathcal{V}_{n-1})}).
	\end{eqnarray*}
	We set $d_{n-1}:=\mathrm{rk}(\mathcal{V}_{n-1})$. Now consider the Euler sequence for $\mathbb{P}=\mathbb{P}^{n}$
	\begin{eqnarray*}
		0\longrightarrow \mathcal{O}_{\mathbb{P}}\longrightarrow \mathcal{O}_{\mathbb{P}}(1)^{\oplus (n+1)}\longrightarrow \mathcal{T}_{\mathbb{P}}\longrightarrow 0
	\end{eqnarray*}
	and tensor this sequence with $\mathcal{O}_{\mathbb{P}}(n-1)^{\oplus d_{n-1}}$. By setting $V_1=\mathrm{Hom}(\mathcal{O}_{\mathbb{P}}(n-1),\mathcal{O}_{\mathbb{P}}(n))^*$, we obtain
	\begin{eqnarray}
		0\longrightarrow \mathcal{O}_{\mathbb{P}}(n-1)^{\oplus d_{n-1}}\longrightarrow V_1^{\oplus d_{n-1}}\otimes \mathcal{O}_{\mathbb{P}}(n)^{\oplus (n+1)}\longrightarrow \mathcal{T}_{\mathbb{P}}^{\oplus d_{n-1}}(n-1)\longrightarrow 0.
	\end{eqnarray}
	Therefore, $R_{\mathcal{O}_{\mathbb{P}}(n)}\mathcal{O}_{\mathbb{P}}(n-1)^{\oplus d_{n-1}}\simeq \mathcal{T}_{\mathbb{P}}^{\oplus d_{n-1}}(n-1)$. The exact sequence (3) descends to the sequence 
	\begin{eqnarray*}
		0\longrightarrow \mathcal{V}_{n-1}\longrightarrow \mathcal{V}_1^{\oplus a_1}\otimes \mathcal{V}_{n-1}\longrightarrow \mathcal{T}_{X}\otimes \mathcal{V}_{n-1}\longrightarrow 0.
	\end{eqnarray*}
	Since $\mathrm{rk}(\mathcal{V}_1^{\oplus a_1}\otimes \mathcal{V}_{n-1})=d_{n-1}\cdot a_1\cdot \mathrm{ind}(A)$ and since $\mathrm{rk}(\mathcal{V}_n)=\mathrm{ind}(A^{\otimes n})=:d_n$ divides $\mathrm{ind}(A)$ (see \cite{SAV}, Theorem 5.5), there is an positive integer $s$ such that $d_n\cdot s=\mathrm{ind}(A)$. From Proposition 4.7, we conclude $\mathcal{V}_1^{\oplus a_1}\otimes \mathcal{V}_{n-1}\simeq \mathcal{V}_n^{\oplus (s\cdot d_{n-1}\cdot a_1)}$. This finally implies $^{\vee}\mathcal{V}_1= R_{\mathcal{V}_{n}}\mathcal{V}_{n-1}\simeq \mathcal{T}_{X}\otimes \mathcal{V}_{n-1}$. Let us now calculate $^{\vee}\mathcal{V}_2=R_{\mathcal{V}_{n}}R_{\mathcal{V}_{n-1}}\mathcal{V}_{n-2}$. We first calculate $R_{\mathcal{V}_{n-1}}\mathcal{V}_{n-2}$. For this, we base change and obtain $R_{\langle\mathcal{V}_{n-1}\rangle_L}(\mathcal{V}_{n-2})_L\simeq R_{\mathcal{O}_{\mathbb{P}}(n-1)}\mathcal{O}_{\mathbb{P}}(n-2)^{\oplus \mathrm{rk}(\mathcal{V}_{n-2})}$. We set $d_{n-2}:=\mathrm{rk}(\mathcal{V}_{n-2})$. Again, tensoring the Euler sequence on $\mathbb{P}$ with $\mathcal{O}_{\mathbb{P}}(n-2)^{\oplus d_{n-2}}$ gives
	\begin{eqnarray*}
		0\longrightarrow \mathcal{O}_{\mathbb{P}}(n-2)^{\oplus d_{n-2}}\longrightarrow V^{\oplus d_{n-2}}\otimes \mathcal{O}_{\mathbb{P}}(n-1)^{\oplus (n+1)}\longrightarrow \mathcal{T}_{\mathbb{P}}^{\oplus d_{n-2}}(n-2)\longrightarrow 0,
	\end{eqnarray*}
	where $V_2=\mathrm{Hom}(\mathcal{O}_{\mathbb{P}}(n-2),\mathcal{O}_{\mathbb{P}}(n-1))^*$.
	This exact sequence descends to 
	\begin{eqnarray*}
		0\longrightarrow \mathcal{V}_{n-2}\longrightarrow \mathcal{V}_1^{\oplus a_1}\otimes\mathcal{V}_{n-2}\longrightarrow \mathcal{T}_X\otimes\mathcal{V}_{n-2}\longrightarrow 0
	\end{eqnarray*}
	on $X$. Therefore, we have $R_{\mathcal{O}_{\mathbb{P}}(n-1)}\mathcal{O}_{\mathbb{P}}(n-2)^{\oplus d_{n-2}}\simeq \mathcal{T}_{\mathbb{P}}^{\oplus d_{n-2}}(n-2)$ and Proposition 4.7 yields $R_{\mathcal{V}_{n-1}}\mathcal{V}_{n-2}\simeq \mathcal{T}_X\otimes\mathcal{V}_{n-2}$. Now we have to calculate $R_{\mathcal{V}_{n}}(\mathcal{T}_X\otimes\mathcal{V}_{n-2})$. After base change, we have $R_{\langle\mathcal{V}_{p-1}\rangle_L}(\mathcal{T}_X\otimes\mathcal{V}_{n-2})_L\simeq R_{\mathcal{O}_{\mathbb{P}}(n)}(\mathcal{T}_{\mathbb{P}}(n-2)^{\oplus d_{n-2}})$. The second external power of the the Euler sequence on $\mathbb{P}$ is
	\begin{eqnarray*}
		0\longrightarrow \mathcal{T}_{\mathbb{P}}\longrightarrow \wedge^2(\mathcal{O}_{\mathbb{P}}(1)^{\oplus (n+1)})\longrightarrow \wedge^2\mathcal{T}_{\mathbb{P}}\longrightarrow 0.
	\end{eqnarray*}
	Tensoring with $\mathcal{O}_{\mathbb{P}}(n-2)^{\oplus d_{n-2}}$ gives 
	\begin{eqnarray*}
		0\longrightarrow \mathcal{T}_{\mathbb{P}}(n-2)^{\oplus d_{n-2}}\longrightarrow (\wedge^2V\otimes \mathcal{O}_{\mathbb{P}}(n))^{\oplus d_{n-2}}\longrightarrow (\wedge^2\mathcal{T}_{\mathbb{P}}(n-2))^{\oplus d_{n-2}}\longrightarrow 0, 
	\end{eqnarray*}
	where $V=\mathrm{Hom}(\mathcal{O}_{\mathbb{P}},\mathcal{O}_{\mathbb{P}}(1))^*$. Hence $R_{\mathcal{O}_{\mathbb{P}}(n)}(\mathcal{T}_{\mathbb{P}}(n-2)^{\oplus d_{n-2}})\simeq (\wedge^2\mathcal{T}_{\mathbb{P}}(n-2))^{\oplus d_{n-2}}$. Note that the latter exact sequence descends to the sequence
	\begin{eqnarray*}
		0\longrightarrow \mathcal{T}_X\otimes\mathcal{V}_{n-2}\longrightarrow (\wedge^2(\mathcal{V}_1^{\oplus a_1}))\otimes \mathcal{V}_{n-2}\longrightarrow (\wedge^2\mathcal{T}_X)\otimes\mathcal{V}_{n-2}\longrightarrow 0.
	\end{eqnarray*}
	So by Proposition 4.7 we obtain $R_{\mathcal{V}_{n}}(\mathcal{T}_X\otimes\mathcal{V}_{n-2})\simeq (\wedge^2\mathcal{T}_X)\otimes\mathcal{V}_{n-2}$. This completes the proof.
\end{proof}
\begin{rema}
	\textnormal{Proposition 4.8 gives another full weak exceptional collection on the considered Brauer--Severi variety. To our best knowledge, this collection is new in the literature.}
\end{rema}
\begin{cor}
	Let $X$ be as above. Then for any coherent sheaf $\mathcal{F}$ there is a spectral sequence
	\begin{eqnarray}
		E^{p,q}_1=\mathrm{Ext}^q(\wedge^{-p}\mathcal{T}_X\otimes \mathcal{V}_{n+p},\mathcal{F})\otimes \mathcal{V}_{n+p}\Rightarrow E^{p+q}=\begin{cases}\mathcal{F}& \text{if} \;\;\;\;\; p+q=0\\
			0& \text{otherwise}
		\end{cases}
	\end{eqnarray}
	The grading is bounded by $0\leq q\leq n$ and $-n\leq p\leq 0$.
\end{cor}

\section{Ulrich bundles on Brauer--Severi varieties}
Let $X$ denote a Brauer--Severi variety over an arbitrary field $k$. Recall that $\mathrm{Pic}(X)$ is generated by $\mathcal{O}_X(p)$, where $p$ denotes the period of $X$. Recall also from Theorem 2.1 that $X$ is embedded by $\mathcal{O}_X(pd)$ and that after base change to some splitting field $L$ of $X$, this morphism becomes the $pd$-uple embedding of $X\otimes_k L\simeq \mathbb{P}^n$.
\begin{prop}
	Let $X$ be a Brauer--Severi variety of dimension $n$ and period $p$. Then $(X,\mathcal{O}_X(p))$ carries an Ulrich bundle of rank $s\cdot n!$ for a suitable $s>0$. Moreover, the integer $s$ can be chosen such that the period $p$ divides $s$.
\end{prop}
\begin{proof}
	Let $E$ be a finite separable splitting field of degree $\mathrm{ind}(X)$ and $L$ its Galois closure. Consider the projection $\pi\colon X\otimes_k L\rightarrow X$ which is finite and surjective. Since $X\otimes_k L\simeq \mathbb{P}_L^n$, there exists an Ulrich bundle $\mathcal{E}$ of rank $n!$ for $(X\otimes_k L,\pi^*\mathcal{O}_X(p))$ according to \cite{BV}, Proposition 3.1. Now Lemma 3.2 provides us with an Ulrich bundle $\pi_*\mathcal{E}$ for $(X,\mathcal{O}_X(p))$. Let us determine the rank of $\pi_*\mathcal{E}$. If $s=[L:k]$, then $\pi^*\pi_*\mathcal{E}\simeq \bigoplus_{g\in G}g^*\mathcal{E}$ for $G=\mathrm{Gal}(L|k)$. This implies $\mathrm{rk}(\pi_*\mathcal{E})=s\cdot n!$. This proves the first statement. Now, as $L$ is a finite Galois extension of $k$, we have that $[E:k]$ divides $s$. The second statement follows from the fact that the period $p$ divides the index $[E:k]$.  
\end{proof}
\begin{prop}
	Let $X$ be a Brauer--Severi variety of dimension $n$ and period $p$. Then for any $d\geq 1$, $(X,\mathcal{O}_X(pd))$ carries an Ulrich bundle of rank $s\cdot (n!)^2$ for a suitable $s>0$. Moreover, the integer $s$ can be chosen such that the period $p$ divides $s$.
\end{prop}
\begin{proof}
	From Proposition 5.1 we get an Ulrich bundle $\mathcal{E}$ of rank $s\cdot n!$ for $(X,\mathcal{O}_X(p))$. Note that there exists a finite surjective projection $\pi\colon X\rightarrow \mathbb{P}^n$. According to \cite{BV}, Proposition 3.1 there is an Ulrich bundle $\mathcal{F}$ of rank $n!$ for $(\mathbb{P}^n,\mathcal{O}_{\mathbb{P}^n}(d))$. Now \cite{ESWV}, Proposition 5.4 states that $\mathcal{E}\otimes \pi^*\mathcal{F}$ is an Ulrich bundle for $(X,\mathcal{O}_X(pd))$. Its rank is $s\cdot (n!)^2$. The second statement follows as in Proposition 5.1
\end{proof}
\begin{prop}
	Let $X$ be a non-split Brauer--Severi variety of period $p$. There is no rank one Ulrich bundle for $(X,\mathcal{O}_X(pd))$ for all $d\geq 1$. 
\end{prop}
\begin{proof}
	Assume $\mathcal{E}$ is an Ulrich bundle for $(X,\mathcal{O}_X(pd))$. Proposition 3.3 states that $\bar{\mathcal{E}}:=\mathcal{E}\otimes_k \bar{k}$ must be an Ulrich bundle for $(\mathbb{P}^n,\mathcal{O}_{\mathbb{P}^n}(pd))$. Now \cite{LV}, Proposition 3.1 implies $pd=1$ or $n=1$ and $\bar{\mathcal{E}}\simeq \mathcal{O}_{\mathbb{P}^1}(pd-1)$. As $X$ is non-split, $pd=1$ is not possible. Moreover, the bundle $\mathcal{O}_{\mathbb{P}^1}(pd-1)$ does not descend to $X$. Indeed, if $\mathcal{O}_{\mathbb{P}^1}(pd-1)$ would descend, the bundle $\mathcal{O}_{\mathbb{P}^1}(pd-1)\otimes \mathcal{O}_{\mathbb{P}^1}(-pd)\simeq \mathcal{O}_{\mathbb{P}^1}(-1)$ would exist on $X$. Since $X$ is non-split, this is impossible.
\end{proof}
Note that a non-split Brauer--Severi curve $C$ has period two. Its corresponding central simple algebra is a quaternion algebra.
\begin{prop}
	Let $C$ be a non-split Brauer--Severi curve. Then for all $d\geq 1$, $\mathcal{E}$ is a rank $r$ Ulrich bundle for $(C,\mathcal{O}(2d))$ if and only if $r=2m$ and $\mathcal{E}\simeq \mathcal{V}^{\oplus m}_{2d-1}$.
\end{prop}
\begin{proof}
	Let $D$ be the quaternion algebra associated to $C$. Then $\mathrm{rk}(\mathcal{V}_i)=\mathrm{ind}(D^{\otimes i})$. Since the index of $D$ is two, we obtain $\mathrm{ind}(D^{\otimes i})=2$ for $i$ odd and $\mathrm{ind}(D^{\otimes i})=1$ otherwise. Now let $\mathcal{E}$ be a rank $r$ Ulrich bundle for $(C,\mathcal{O}(2d))$. By Proposition 3.3, $\bar{\mathcal{E}}$ is an Ulrich bundle for $(\mathbb{P}^1,\mathcal{O}_{\mathbb{P}}(2d))$. According to \cite{LV}, Corollary 4.5, the bundle $\bar{\mathcal{E}}$ is Ulrich if and only if $\bar{\mathcal{E}}\simeq \mathcal{O}_{\mathbb{P}}(2d-1)^{\oplus r}$. If $r=2m$, we see that $\mathcal{V}^{\oplus m}_{2d-1}$ must be Ulrich, since it base changes to $\mathcal{O}_{\mathbb{P}}(2d-1)^{\oplus r}$. On the other hand, if $\mathcal{E}$ is Ulrich, we have $\bar{\mathcal{E}}\simeq \mathcal{O}_{\mathbb{P}}(2d-1)^{\oplus r}$ for a suitable $r$. Now see \cite{NV} to conclude that $\mathcal{O}_{\mathbb{P}}(2d-1)^{\oplus r}$ does not descend to $C$ for odd $r$. In fact, this follows also from taking the determinant $\mathrm{det}(\mathcal{O}_{\mathbb{P}}(2d-1)^{\oplus r})$. We have $\mathrm{det}(\mathcal{O}_{\mathbb{P}}(2d-1)^{\oplus r})\simeq \mathcal{O}_{\mathbb{P}}(r\cdot(2d-1))$. But if $r$ is odd, $\mathcal{O}_{\mathbb{P}}(r\cdot(2d-1))$ does not descend to $C$ since $\mathrm{Pic}(C)$ is generated by $\mathcal{O}_C(2)$. For even $r$, however, $\mathcal{O}_{\mathbb{P}}(2d-1)^{\oplus r}$ descends uniquely (up to isomorphism) to $\mathcal{V}^{\oplus r}_{2d-1}$.
\end{proof}
\begin{cor}
	Let $C$ be a non-split Brauer--Severi curve. Then the smallest possible rank of an Ulrich bundle for $(C,\mathcal{O}(2d))$ is two.
\end{cor}
\begin{rema}
	\textnormal{Corollary 5.5 amounts to saying that any Brauer--Severi curve over an arbitrary field $k$ which admits an Ulrich line bundle must admit a $k$-rational point.}
\end{rema}
Let $C$ be a Brauer--Severi curve over $k$. Consider the symmetric power $S^m(C)$, i.e the quotient of $\prod^m_{i=1} C$ by the symmetric group $S_m$, and note that this variety is non-singular. Since $S^m(C)\otimes_k \bar{k}\simeq S^m(\mathbb{P}^1_{\bar{k}})\simeq \mathbb{P}^m_{\bar{k}}$, we conclude that $S^m(C)$ must be a Brauer--Severi variety. Note that $\mathrm{ind}(S^m(C))\leq 2$. 
\begin{prop}
	If $\mathrm{ind}(S^m(C))=1$ or $\mathrm{per}(S^m(C))=1$, there is always an Ulrich bundle of rank $m!$ for $(S^m(C),\mathcal{O}(d))$. If $\mathrm{per}(S^m(C))=2$, there is always an Ulrich bundle of rank $2^m\cdot m!$ for $(S^m(C),\mathcal{O}(2d))$.
\end{prop}
\begin{proof}
	If $\mathrm{ind}(S^m(C))=1$ or $\mathrm{per}(S^m(C))=1$, the Brauer--Severi variety $S^m(C)$ is isomorphic to $\mathbb{P}^m$ and the assertion is the content of \cite{BV}, Proposition 3.1. So let us assume $\mathrm{per}(S^m(C))=2$. This means that $\mathcal{O}(2)$ is a generator of $\mathrm{Pic}(S^m(C))$. Consider the quotient map $\pi\colon C^m\rightarrow S^m(C)$. The pullback $\pi^*\mathcal{O}(2)$ is $\mathcal{O}_C(2)\boxtimes\cdots \boxtimes \mathcal{O}_C(2)$. Proposition 5.4 and Lemma 3.1 imply that the bundle $\mathcal{E}=\mathcal{V}_{2d-1}\boxtimes\cdots\boxtimes \mathcal{V}_{2dm-1}$ is an Ulrich bundle for $(C^m,\pi^*\mathcal{O}(2d))$. From Lemma 3.2 we conclude that $\pi_*\mathcal{E}$ is an Ulrich bundle for $(S^m(C),\mathcal{O}(2d))$. 
\end{proof}
To prove Theorem 5.10 below, we need the next two results.
\begin{thm}[\cite{ESWV}, Theorem 5.1]
	Let $\mathcal{E}$ be an vector bundle on $\mathbb{P}^n$. Then $\mathcal{E}$ is an Ulrich bundle for $(\mathbb{P}^n,\mathcal{O}_{\mathbb{P}^n}(d))$ if and only if 
	\begin{eqnarray*}
		h^q(\mathcal{E}\otimes\mathcal{O}_{\mathbb{P}^n}(i))\neq 0\Leftrightarrow \begin{cases}q=0 & \text{and} \;\;\;\;\; -d<i,\\
			0<q<n & \text{and} \;\;\;\;\; -(q+1)d<i< -qd,\\
			q=n & \text{and} \;\;\;\;\;\; i< -nd.
		\end{cases}
	\end{eqnarray*}
	All $h^q(\mathcal{E}(i))$ are determined by $\chi(\mathcal{E}(i))=\frac{\mathrm{rk}(\mathcal{E})}{n!}(i+d)\cdots (i+nd)$.
\end{thm}

\begin{thm}[\cite{LV}, Theorem 4.1]
	If $\mathcal{E}$ is an Ulrich bundle for $(\mathbb{P}^n,\mathcal{O}_{\mathbb{P}^n}(d))$ and $1\leq j\leq n$, then
	\begin{eqnarray*}
		h^q(\mathcal{E}\otimes\mathcal{O}_{\mathbb{P}^n}(i)\otimes \Omega_{\mathbb{P}^n}^j\otimes \mathcal{O}_{\mathbb{P}^n}(j))\neq 0\Rightarrow \begin{cases}q=0 & \text{and} \;\;\;\;\; -d<i,\\
			0<q<n & \text{and} \;\;\;\;\; -(q+1)d<i\leq -qd,\\
			q=n & \text{and} \;\;\;\;\;\; i\leq -nd.
		\end{cases}
	\end{eqnarray*}
\end{thm}

\begin{thm}
	Let $n\neq 1$ and $X$ a $n$-dimensional Brauer--Severi variety of period $p$. Set $d_n=\mathrm{rk}(\mathcal{V}_n)$. If $\mathcal{E}$ is an Ulrich bundle for $(X,\mathcal{O}_X(pd))$, then there is an exact sequence
	\begin{eqnarray}
		0\rightarrow \mathcal{V}^{\oplus a_{-n}}_{-n}\rightarrow \cdots \rightarrow \mathcal{V}^{\oplus a_{-n+1}}_{-n+1}\rightarrow \mathcal{V}^{\oplus a_{-1}}_{-1}\rightarrow \mathcal{E}\otimes\mathcal{O}_X(-pd)^{\oplus d^2_n}\rightarrow 0.
	\end{eqnarray}
\end{thm}
\begin{proof}
	By Corollary 4.10 there is a spectral sequence 
	\begin{eqnarray}
		E^{l,m}_1=\mathrm{Ext}^m(\wedge^{-l}\mathcal{T}_X\otimes \mathcal{V}_{n+l},\mathcal{F})\otimes \mathcal{V}_{n+l}\Rightarrow E^{l+m}=\begin{cases}\mathcal{F}& \text{if} \;\;\;\;\; l+m=0\\
			0& \text{otherwise}
		\end{cases}
	\end{eqnarray}
	where the grading is bounded by $0\leq m\leq n$ and $-n\leq l\leq 0$. Note that $\mathcal{V}^{\vee}_i\simeq \mathcal{V}_{-i}$ (see \cite{NV}, Proposition 5.4). Therefore, we have
	\begin{eqnarray*}
		\mathrm{Ext}^m(\wedge^{-l}\mathcal{T}_X\otimes \mathcal{V}_{n+l},\mathcal{F})\simeq H^m(X, \mathcal{F}\otimes \Omega_X^{-l}\otimes \mathcal{V}_{-n-l}).
	\end{eqnarray*}
	After base change to $\bar{k}$, we find
	\begin{eqnarray*}
		H^m(X, \mathcal{F}\otimes \Omega_X^{-l}\otimes \mathcal{V}_{-n-l})\otimes_k \bar{k}\simeq H^m(\mathbb{P}^{n},\bar{\mathcal{F}}\otimes \Omega^{-l}_{\mathbb{P}^{n}}(-n-l))^{\oplus \mathrm{rk}( \mathcal{V}_{-n-l})}.
	\end{eqnarray*}
	Note that 
	\begin{eqnarray*}
		H^m(\mathbb{P}^{n},\bar{\mathcal{F}}\otimes \Omega^{-l}_{\mathbb{P}^{n}}(-n-l))\simeq H^m(\mathbb{P}^{n},\bar{\mathcal{F}}(-n)\otimes \Omega^{-l}_{\mathbb{P}^{n}}(-l)). 
	\end{eqnarray*}
	Now consider the bundle $\mathcal{F}=\mathcal{E}\otimes\mathcal{V}_{n}\otimes\mathcal{O}_X(-pd)$ on $X$. Then
	\begin{eqnarray*}
		H^m(\mathbb{P}^{n},\bar{\mathcal{F}}(-n)\otimes \Omega^{-l}_{\mathbb{P}^{n}}(-l))\simeq H^m(\mathbb{P}^{n}, \bar{\mathcal{E}}\otimes\mathcal{O}_{\mathbb{P}^n}(-pd)\otimes\Omega^{-l}_{\mathbb{P}^{n}}(-l))^{\oplus d_n}
	\end{eqnarray*}
	For $n\neq 1$, Theorem 5.9 implies $h^m(\bar{\mathcal{E}}(-pd)\otimes \Omega^{-l}_{\mathbb{P}^{n}}(-l))=0$ for $m\neq 1$ and $-n\leq l\leq-1$. Furthermore, Theorem 5.8 implies $h^m(\bar{\mathcal{E}}(-pd))^{\oplus d_n}=0$ for $m\geq 0$. In particular, we have 
	\begin{eqnarray*}
		\mathrm{Ext}^m(\wedge^{-l}\mathcal{T}_X\otimes \mathcal{V}_{n+l},\mathcal{F})=0\;\; \text{for}\;\; m\neq 1\;\; \text{and}\;\;
		\mathrm{Ext}^1(\mathcal{V}_n,\mathcal{F})=0.
	\end{eqnarray*}
	By the properties of spectral sequence, we have $E_2^{l,1}=E_{\infty}^{l,1}=0$ for $l\neq -1$ and $E^{-1,1}_2=E^{-1,1}_{\infty}=\mathcal{F}$. This gives the following exact sequence
	\begin{eqnarray}
		0\longrightarrow \mathcal{V}_{0}^{\oplus b_0}\longrightarrow \mathcal{V}_{1}^{\oplus b_{1}}\longrightarrow \cdots \longrightarrow \mathcal{V}_{n-1}^{\oplus b_{n-1}}\longrightarrow \mathcal{E}\otimes\mathcal{V}_{n}\otimes\mathcal{O}_X(-pd)\longrightarrow 0.
	\end{eqnarray}
	Note that $(\mathcal{V}_{n}\otimes\mathcal{V}_{n}^{\vee})\otimes_k \bar{k}\simeq \mathcal{O}_{\mathbb{P}^n}^{\oplus d^2_{n}}$. From Proposition 4.7, we conclude $\mathcal{V}_{n}\otimes\mathcal{V}_{n}^{\vee}\simeq \mathcal{O}_{X}^{\oplus d^2_{n}}$. Tensoring the sequence (7) with $\mathcal{V}^{\vee}_{n}\simeq \mathcal{V}_{-n}$, we find
	\begin{eqnarray*}
		0\longrightarrow \mathcal{V}_{-n}^{\oplus a_{-n}}\longrightarrow \mathcal{V}_{-n+1}^{\oplus a_{-n+1}}\longrightarrow \cdots \longrightarrow \mathcal{V}_{-1}^{\oplus a_{-1}}\longrightarrow \mathcal{E}\otimes\mathcal{O}_X(-pd)^{\oplus d^2_{n}}\longrightarrow 0.
	\end{eqnarray*}
	Here $a_{-r}=d_n\cdot b_{n-r}$ for $1\leq r\leq n$ with $b_s=\mathrm{ext}^1((\wedge^{n-s}\mathcal{T}_X\otimes \mathcal{V}_{s},\mathcal{F}))$, $0\leq s\leq n$, and $\mathcal{F}=\mathcal{E}\otimes\mathcal{V}_{n}\otimes\mathcal{O}_X(-pd)$. This completes the proof.
\end{proof}
\begin{rema}
	\textnormal{Theorem 5.10 is a direct generalization of \cite{LV}, Theorem 4.3}
\end{rema}
\begin{thm}
	Let $X$ be a non-split Brauer--Severi variety of dimension $p-1$, where $p$ denotes an odd prime, and $\mathcal{E}$ an Ulrich bundle for $(X,\mathcal{O}_X(p))$. Then there are integers $c_{p-1-j}>0$, $j=1,...,p-2$, such that
	\begin{eqnarray*}
		\mathrm{rk}(\mathcal{E})\cdot ((p-1)!+\prod^{p-2}_{i=1}{(ip-1)})=(p\cdot(p-1)!)(\sum^{p-2}_{j=1}{(-1)^{j+1}c_{p-1-j}}).
	\end{eqnarray*}
	In particular, $p$ divides $\mathrm{rk}(\mathcal{E})\cdot ((p-1)!+\prod^{p-2}_{i=1}{(ip-1)})$.
\end{thm}
\begin{proof}
	Let $\mathcal{E}$ be an Ulrich bundle for $(X,\mathcal{O}_X(p))$ . Theorem 5.10 above gives us an exact sequence
	\begin{eqnarray*}
		0\rightarrow \mathcal{V}^{\oplus b_{0}}_{0}\rightarrow \mathcal{V}^{\oplus b_{1}}_{1}\rightarrow\cdots \rightarrow \mathcal{V}^{\oplus b_{p-2}}_{p-2}\rightarrow \mathcal{E}\otimes\mathcal{V}_{p-1}\otimes \mathcal{O}_X(-p)\rightarrow 0.
	\end{eqnarray*}
	One can show $b_i=p\cdot c_i$ for suitable positive integer $c_i>0$. In fact, the proof of Theorem 5.10 shows $b_0=\mathrm{ext}^1(\wedge^{p-1}\mathcal{T}_X,\mathcal{F})$ for $\mathcal{F}=\mathcal{E}\otimes \mathcal{V}_{p-1}\otimes\mathcal{O}_X(p)$. After base change to $\bar{k}$ we obtain $\mathrm{Ext}^1(\wedge^{p-1}\mathcal{T}_X,\mathcal{F})\otimes_k\bar{k}\simeq H^1(\mathbb{P}^{p-1},\bar{\mathcal{E}}(-p-1))^{\oplus \mathrm{rk}(\mathcal{V}_{p-1})}$. Recall that $\mathrm{rk}(\mathcal{V}_l)=\mathrm{ind}(A^{\otimes l})$ (see \cite{NV}, Corollary 6.4). Since the period equals index, \cite{NV}, Proposition 6.8 implies $\mathrm{ind}(A^{\otimes l})=p/(p,l)$, where $(p,l)$ denotes the greatest common divisior of $p$ and $l$. From \cite{NV}, Proposition 5.4 we conclude $\mathcal{V}_l^{\vee}\simeq \mathcal{V}_{-l}$. Hence $\mathrm{rk}(\mathcal{V}_l)=\mathrm{rk}(\mathcal{V}_{-l})$. Since $1\leq l\leq p-1$, we get $\mathrm{rk}(\mathcal{V}_l)=p$. Therefore, $b_0=h^1(\mathbb{P}^{p-1},\bar{\mathcal{E}}(-p-1))^{\oplus p}$. So if we set $c_0:=h^1(\bar{\mathcal{E}}(-p-1))$, we get $b_0=p\cdot c_0$. In the same way one shows $b_i=p\cdot c_i$, $i=1,...,p-1$. 
	Since $p$ is an odd prime, we get
	\begin{eqnarray*}
		\mathrm{rk}(\mathcal{E})\cdot p &= &\sum^{p-1}_{j=1}(-1)^{j+1}\mathrm{rk}(\mathcal{V}_{p-1-j})\cdot b_{p-1-j}\\
		&=& -p\cdot c_0+\sum^{p-2}_{j=1}(-1)^{j+1}\mathrm{rk}(\mathcal{V}_{p-1-j})\cdot b_{p-1-j}\\
		&=&  -p\cdot c_0+\sum^{p-2}_{j=1}(-1)^{j+1}\mathrm{rk}(\mathcal{V}_{p-1-j})\cdot p\cdot c_{p-1-j}.
	\end{eqnarray*}
	Now we divide by $p$ and get 
	\begin{eqnarray*}
		\mathrm{rk}(\mathcal{E})+c_0=\sum^{p-2}_{j=1}(-1)^{j+1}\mathrm{rk}(\mathcal{V}_{p-1-j})\cdot c_{p-1-j}.
	\end{eqnarray*}
	Since $\mathrm{rk}(\mathcal{V}_{p-1-j})=p$ for $j=1,...,p-1$, we have  
	\begin{eqnarray*}
		\mathrm{rk}(\mathcal{E})+c_0=p\cdot(\sum^{p-2}_{j=1}(-1)^{j+1}c_{p-1-j}).
	\end{eqnarray*}
	Note that $c_0=h^1(\bar{\mathcal{E}}(-p-1))$. With the help of Theorem 5.8 we calculate $\chi(\bar{\mathcal{E}}(-p-1))=-h^1(\bar{\mathcal{E}}(-p-1))$. This gives 
	\begin{eqnarray*}
		\chi(\bar{\mathcal{E}}(-p-1))=-h^1(\bar{\mathcal{E}}(-p-1))=-c_0=\frac{\mathrm{rk}(\mathcal{E})}{(p-1)!}(-p-1+p)\cdots (-p-1+(p-1)p)
	\end{eqnarray*}
	and therefore
	\begin{eqnarray*}
		\mathrm{rk}(\mathcal{E})+c_0&=&\mathrm{rk}(\mathcal{E})+\frac{\mathrm{rk}(\mathcal{E})}{(p-1)!}(p-1)\cdots ((p-2)p-1)\\
		&=&p\cdot(\sum^{p-2}_{i=1}(-1)^{j+1}c_{p-1-j}).
	\end{eqnarray*}
	The equality 
	\begin{eqnarray*}
		\mathrm{rk}(\mathcal{E})+\frac{\mathrm{rk}(\mathcal{E})}{(p-1)!}(p-1)\cdots ((p-2)p-1)
		=p\cdot(\sum^{p-2}_{i=1}(-1)^{j+1}c_{p-1-j})
	\end{eqnarray*}
	gives the desired formula. This completes the proof.
\end{proof}
\noindent 
We see that if $p$ does not divide $((p-1)!+\prod^{p-2}_{i=1}{(ip-1)})$, then it must divide the rank of $\mathcal{E}$. 
\begin{cor}
	Let $X$ be a non-split Brauer--Severi variety as in Theorem 5.12. Then the rank of any Ulrich bundle for $(X,\mathcal{O}_X(p))$ must be a multiple of $p$. 
\end{cor}
\begin{proof}
	Let $\mathcal{E}$ be an Ulrich bundle for $(X,\mathcal{O}_X(p))$. Theorem 5.12 states that there are integers $c_{p-1-j}>0$, $j=1,...,p-2$, such that
	\begin{eqnarray*}
		\mathrm{rk}(\mathcal{E})\cdot ((p-1)!+\prod^{p-2}_{i=1}{(ip-1)})=(p\cdot(p-1)!)(\sum^{p-2}_{j=1}{(-1)^{j+1}c_{p-1-j}}).
	\end{eqnarray*}
	In particular, $p$ divides $\mathrm{rk}(\mathcal{E})\cdot ((p-1)!+\prod^{p-2}_{i=1}{(ip-1)})$. Using Wilson's theorem, one can show that $p$ does not divide $((p-1)!+\prod^{p-2}_{i=1}{(ip-1)})$. Therefore, $p$ must divide $\mathrm{rk}(\mathcal{E})$ and hence any Ulrich bundle for $(X,\mathcal{O}_X(p))$ must have rank a multiple of $p$. 
\end{proof}
\begin{rema}
	\textnormal{Propositions 5.1 and 5.2 state that the the rank of the obtained Ulrich bundle is divisible by the period. Now Corollary 5.13 shows that this the case for any Ulrich bundle, i.e that the rank of any Ulrich bundle on a Brauer--Severi variety $X$ is divided by the period, at least if $\mathrm{dim}(X)=p-1$ for an odd prime $p$. We wonder whether this is true in general.}
\end{rema}

\begin{thm}
	Let $X$ be a Brauer--Severi variety of dimension $3$ where $\mathcal{O}_X(2)$ exists. Suppose there is a smooth geometrically connected genus one curve $C$ on $X$ such that the restriction map $H^0(X,\mathcal{O}_X(2))\rightarrow H^0(C,\mathcal{O}_C(2))$ is bijective. Then there is an  unique Ulrich bundle of rank two for $(X,\mathcal{O}_X(2))$.  
\end{thm}
\begin{proof}
	Since $\bar{C}:=C\otimes_k \bar{k}$ is an elliptic curve, we have $\omega_{\bar{C}}=\mathcal{O}_{\bar{C}}$. Now Proposition 4.7 implies $\omega_C=\mathcal{O}_C$. Let $\mathcal{N}:=\mathcal{N}_{C/X}$ be the normal bundle. From adjunction formula we obtain  $\mathrm{det}(\mathcal{N})\simeq \omega^{-1}_{X|C}$. Hence, there is a line bundle $\mathcal{L}:=\omega^{-1}_X$ extending $\mathrm{det}(\mathcal{N})$. Setting $\bar{\mathcal{N}}=\mathcal{N}_{\bar{C}/\bar{X}}$, we have $\mathrm{det}(\bar{\mathcal{N}})\simeq \omega^{-1}_{\mathbb{P}^3|\bar{C}}$. Hence the line bundle $\bar{\mathcal{L}}=\mathcal{L}\otimes_k \bar{k}$ extends $\mathrm{det}(\bar{\mathcal{N}})$. Note that $H^2(X,\bar{\mathcal{L}}^{-1})=H^2(\mathbb{P}^3,\omega_{\mathbb{P}^3})=0$. Since $\bar{C}$ is a local complete intersection in $\mathbb{P}^3$ we can apply \cite{BAV}, Theorem 2.2 (generalized Hartshorne--Serre correspondence) to obtain a rank two vector bundle $\mathcal{E}'$ on $X\otimes_k\bar{k}=\mathbb{P}^3$ sitting in the following short exact sequence
	\begin{eqnarray}
		\begin{xy}
			\xymatrix{
				0\ar[r] &\mathcal{O}_{\bar{X}}\ar[r]^t &\ar[r]\mathcal{E}' & \mathcal{I}_{\bar{C}}\otimes \bar{\mathcal{L}}  \ar[r]& 0.\\
			}
		\end{xy}
	\end{eqnarray}
	Since $H^1(\bar{X},\omega_{\bar{X}})=0$, the pair $(t,\mathcal{E}')$ is unique up to isomorphism, i.e. 
	\begin{eqnarray*}
		\mathrm{Ext}^1(\mathcal{I}_{\bar{C}}\otimes \bar{\mathcal{L}}, \mathcal{O}_{\bar{X}})\simeq \bar{k}.
	\end{eqnarray*}
	Since the coherent sheaves $\mathcal{O}_X$ and $\mathcal{I}_C\otimes\mathcal{L}$ exist on $X$, we also have $\mathrm{Ext}^1(\mathcal{I}_C\otimes\mathcal{L}, \mathcal{O}_X)\simeq k$. Therefore, we have an exact sequence 
	\begin{eqnarray}
		\begin{xy}
			\xymatrix{
				0\ar[r] &\mathcal{O}_{X}\ar[r]^s &\ar[r]\mathcal{E} & \mathcal{I}_{C}\otimes \mathcal{L}  \ar[r]& 0\\
			}
		\end{xy}
	\end{eqnarray}
	that base changes to the exact sequence (8). Then Proposition 4.7 yields $\bar{\mathcal{E}}:=\mathcal{E}\otimes_k \bar{k}\simeq \mathcal{E}'$.
	Moreover, since $\mathrm{det}(\bar{\mathcal{E}})=\mathcal{L}\otimes_k \bar{k}=\mathcal{O}_{\mathbb{P}^3}(4)$ (see \cite{BAV}, Theorem 2.2), we conclude $\mathrm{det}(\mathcal{E})=\mathcal{L}=\mathcal{O}_{X}(4)$. Let us show that $\mathcal{E}$ is an Ulrich bundle. Note that $X$ is embedded into $\mathbb{P}^9$ via $\mathcal{O}_X(2)$, i.e. there is the morphism $\phi_2\colon X\rightarrow \mathbb{P}^9$ from Theorem 2.1. Recall that a rank two bundle $\mathcal{F}$ is \emph{special} in the sense of \cite{ESWV} if $\mathrm{det}(\mathcal{F})=\omega_X\otimes \mathcal{O}(4)$. Here $\mathcal{O}(4)=\phi_2^*\mathcal{O}_{\mathbb{P}^9}(4)$. Since $X\otimes_k \bar{k}\simeq \mathbb{P}^3$, we see that $\bar{\mathcal{E}}$ satisfies $\mathrm{det}(\bar{\mathcal{E}})\simeq \mathcal{O}_{\mathbb{P}^3}(4)$. Hence $\bar{\mathcal{E}}$ is special (see \cite{BV}, Remark right after Proposition 5.6.1). Now $\phi^*\mathcal{O}_{\mathbb{P}^9}(1)=\mathcal{O}_X(2)$, and therefore $\mathrm{det}(\mathcal{E})=\omega_X\otimes \mathcal{O}_X(2\cdot 4)$. Using $\mathcal{F}=\mathcal{F}^{\vee}\otimes\mathrm{det}(\mathcal{F})$ for a rank two vector bundle $\mathcal{F}$, we calculate
	\begin{eqnarray*}
		\mathcal{E}(-2)\simeq \omega_X\otimes \mathcal{E}(-6)^{\vee},\; \mathcal{E}(-2)\simeq \omega_X\otimes \mathcal{E}(-4)^{\vee}\;\;\text{and}\;\;
		\mathcal{E}(-6)\simeq \omega_X\otimes \mathcal{E}(-2)^{\vee}.
	\end{eqnarray*}
	To show that $\mathcal{E}$ is an Ulrich bundle, we have to verify that $H^{\bullet}(X,\mathcal{E}(-2r))=0$ for $1\leq r\leq 3$. By Serre duality, it suffices to prove $H^{\bullet}(X,\mathcal{E}(-2))=0$ and $H^i(X,\mathcal{E}(-4))=0$ for $i=0,1$. We first show $H^{\bullet}(X,\mathcal{E}(-2))=0$. Tensoring the exact sequence (9) with $\mathcal{O}_X(-2)$ gives
	\begin{eqnarray*}
		\begin{xy}
			\xymatrix{
				0\ar[r] &\mathcal{O}_X(-2)\ar[r]&\ar[r]\mathcal{E}(-2) & \mathcal{I}_C\otimes \mathcal{O}_X(2) \ar[r]& 0.\\
			}
		\end{xy}
	\end{eqnarray*}
	The long exact sequence for cohomology shows that it is enough to prove $H^i(X,\mathcal{O}_X(-2))=0$ and $H^i(X,\mathcal{I}_C(2))=0$. The vanishing of $H^i(X,\mathcal{O}_X(-2))$ follows from the vanishing of $H^i(X,\mathcal{O}_X(-2))\otimes_k \bar{k}\simeq H^i(\mathbb{P}^3,\mathcal{O}_{\mathbb{P}^3}(-2))$. The vanishing of $H^i(X,\mathcal{I}_C(2))=0$ can be proved using the exact sequence 
	\begin{eqnarray*}
		\begin{xy}
			\xymatrix{
				0\ar[r] &\mathcal{I}_C(2)\ar[r]&\ar[r]\mathcal{O}_X(2)& \mathcal{O}_C(2)\ar[r]& 0\\
			}
		\end{xy}
	\end{eqnarray*}
	and the fact that we assumed the restriction map $H^0(X,\mathcal{O}_X(2))\rightarrow H^0(C,\mathcal{O}_C(2))$ to be bijective. Note that the bijectivity of $H^0(X,\mathcal{O}_X(2))\rightarrow H^0(C,\mathcal{O}_C(2))$ in particular implies $h^0(\mathcal{I}_C(2))=h^1(\mathcal{I}_C(2))=0$. It remains to show $H^i(X,\mathcal{E}(-4))=0$ for $i=0,1$. Again, using  
	\begin{eqnarray*}
		\begin{xy}
			\xymatrix{
				0\ar[r] &\mathcal{O}_X(-4)\ar[r]&\ar[r]\mathcal{E}(-4) & \mathcal{I}_C\ar[r]& 0,\\
			}
		\end{xy}
	\end{eqnarray*}
	it suffices to show $H^i(X,\mathcal{O}_X(-4))=0$ and $H^i(X,\mathcal{I}_C)=0$ for $i=0,1$. The vanishing of $H^i(X,\mathcal{O}_X(-4))$ follows again from base change to $\bar{k}$ and for the vanishing of $H^i(X,\mathcal{I}_C)$ one uses the exact sequence
	\begin{eqnarray*}
		\begin{xy}
			\xymatrix{
				0\ar[r] &\mathcal{I}_C\ar[r]&\ar[r]\mathcal{O}_X& \mathcal{O}_C\ar[r]& 0\\
			}
		\end{xy}
	\end{eqnarray*}
	and the fact that $H^0(X,\mathcal{O}_X)\simeq k$ and $H^0(C,\mathcal{O}_C)\simeq k$. The uniqueness of $\mathcal{E}$ follows from Proposition 3.3 and \cite{LV}, Proposition 5.4 This completes the proof.
\end{proof}
\begin{cor}
	Let $X$ be as in Theorem 5.15. Then there is a Ulrich bundle of rank $12$ for $(X,\mathcal{O}_X(2d))$ for any $d\geq 1$.
\end{cor}
\begin{proof}
	Let $\mathcal{E}$ be the Ulrich bundle for $(X,\mathcal{O}_X(2))$  from Theorem 5.15. Note that there exists a finite surjective projection $\pi\colon X\rightarrow \mathbb{P}^3$. According to \cite{BV}, Proposition 3.1 there is an Ulrich bundle $\mathcal{F}$ of rank $n!$ for $(\mathbb{P}^n,\mathcal{O}_{\mathbb{P}^n}(d))$. Now \cite{ESWV}, Proposition 5.4 states that $\mathcal{E}\otimes \pi^*\mathcal{F}$ is an Ulrich bundle for $(X,\mathcal{O}_X(pd))$. Its rank is $12$.
\end{proof}
\begin{exam}[genus one curve on Brauer--Severi threefold]
	\textnormal{We want to show that there exists a smooth geometrically connected genus 1 curve $C$ on a non-split Brauer--Severi threefold of index two which satisfies the hypotheses of Theorem 5.15. So let us take a smooth geometrically connected genus 1 curve $C$ over $\mathbb{R}$. Then there is an exact sequence
		\begin{eqnarray*}
			0\longrightarrow \mathrm{Pic}(C)\longrightarrow \mathrm{Pic}_C(\mathbb{R})\stackrel{\delta}{\longrightarrow} \mathrm{Br}(\mathbb{R}).
		\end{eqnarray*}
		Here $\mathrm{Pic}_C$ is the Picard scheme of $C$ with respect to the \'etale topology. With a similar argument as in \cite{LIV}, Example 2.2 one can show that the map
		$\mathrm{Pic}(C)\rightarrow \mathrm{Pic}_C(\mathbb{R})$ is not surjective. Since $\mathrm{Br}(\mathbb{R})\simeq \mathbb{Z}/2\mathbb{Z}$ is generated by the class $\mathbb{H}$ of the Hamilton quaternions, we find that there is a $\mathcal{L}\in \mathrm{Pic}_C(\mathbb{R})$ such that $\delta(\mathcal{L})=\mathbb{H}$. One can choose $\mathcal{L}$ to be of degree $6$. In this case, $\mathcal{L}$ defines an embedding $\phi\colon C\rightarrow P$ into a Brauer--Severi variety $P$ of dimension $5$ (see \cite{LIV}). The class of $P$ corresponds to $\mathrm{Mat}_3(\mathbb{H})$. Since the index of $P$ is two, there is a $0$-cycle $Z$ of degree 2 on $P$. Now one can choose $Z$ in such a way that the linear span $P'$ of $Z$ has no intersection with $\phi(C)$. If we project away from $P'$, we obtain a rational map $P\dashrightarrow Q$. The composition $C\rightarrow P\dashrightarrow Q$ induces an embedding of $C$ into $Q$. Now $Q$ is a Brauer--Severi threefold with class $\mathbb{H}$. The degree of $C$ in $Q$ is still $6$ and we can use a classical result of Castelnuovo to conclude that after base change to $\mathbb{C}$, the elliptic curve $C\otimes_{\mathbb{R}}{\mathbb{C}}$ is not contained in quadric surface. This implies that the restriction map $H^0(Q,\mathcal{O}_Q(2))\rightarrow H^0(C,\mathcal{O}_C(2))$ is bijective.}
\end{exam}
\begin{rema}
	\textnormal{Let $X$ be as above. Then Proposition 5.2 gives us an Ulrich bundle $\mathcal{E}$ for $(X,\mathcal{O}_X(2d))$ which has rank multiple of $36$. Under the assumptions of Theorem 5.15, we see that Corollary 5.16 implies that the minimal rank of an Ulrich bundle is at most $12$.}
\end{rema}
\begin{rema}
	\textnormal{We want to note that the problem of finding smooth genus 1 curves on Brauer--Severi varieties with prescribed properties is quite challenging. Up to now, it it an open problem to find smooth genus 1 curves on arbitrary Brauer--Severi varieties. We refer the interested reader to \cite{DJV} and references therein.}
\end{rema}

{\small MATHEMATISCHES INSTITUT, HEINRICH--HEINE--UNIVERSIT\"AT 40225 D\"USSELDORF, GERMANY}\\
E-mail adress: novakovic@math.uni-duesseldorf.de

\end{document}